\documentclass[11pt]{article}
\usepackage{amsthm,amsfonts, amsbsy, amssymb,amsmath,graphicx}
\usepackage{graphics}
\usepackage[english]{babel}
\usepackage{graphicx}
\usepackage{lineno}
\usepackage{enumerate}
\usepackage{hyperref}
\usepackage{color}

\hypersetup{colorlinks=true}

\hypersetup{colorlinks=true, linkcolor=blue, citecolor=blue,urlcolor=blue}

\input epsf
\usepackage{epic,latexsym,amssymb}
\usepackage{tikz}

\usepackage{tikz,epsf}

\usepackage{amsfonts}
\usepackage{amscd}
\usepackage{amsmath}
\usepackage{graphicx}
\usepackage{color}
\usepackage{caption,subcaption}

\newtheorem{theorem}{Theorem}

\newtheorem{proposition}[theorem]{Proposition}
\newtheorem{lemma}[theorem]{Lemma}

\newtheorem{problem}{Problem}

\newtheorem{corollary}[theorem]{Corollary}

\newtheorem{remark}[theorem]{Remark}

\newcommand{\diam}{{\rm diam}}

\newcommand{\pow}[1]{^{[\natural #1]}}

\newcommand{\opack}{\rho^{\rm o}}

\newcommand{\cart}{\, \Box \,}

\tikzstyle{vertex}=[circle, draw, inner sep=0pt, minimum size=6pt]

\newcommand{\QEDmark}{\mbox{\textsc{qed}}}
\newcommand{\proofStarter}[1]{\textsc{#1}}

\begin{document}

\title{Injective coloring of graphs revisited\vspace{7mm}}
\date{}
\author {
Bo\v{s}tjan Bre\v{s}ar$^{a,b}$, Babak Samadi$^c$ and Ismael G. Yero$^d$\vspace{1.5mm}\\
$^a$Faculty of Natural Sciences and Mathematics, University of Maribor, Slovenia\\
$^b$Institute of Mathematics, Physics and Mechanics, Ljubljana, Slovenia\\
{\tt bostjan.bresar@um.si}\vspace{1.5mm}\\
$^c$Department of Mathematics, Faculty of Mathematical Sciences, Alzahra University,\\ Tehran, Iran\\
{\tt b.samadi@alzahra.ac.ir}\vspace{1.5mm}\\
$^d$Departamento de Matem\'{a}ticas, Universidad de C\'{a}diz, Algeciras, Spain\\
{\tt ismael.gonzalez@uca.es}\vspace{3mm}\\
}
\date{}
\maketitle

\begin{abstract}
An open packing in a graph $G$ is a set $S$ of vertices in $G$ such that no two vertices in $S$ have a common neighbor in $G$. The injective chromatic number $\chi_i(G)$ of $G$ is the smallest number of colors assigned to vertices of $G$ such that each color class is an open packing. Alternatively, the injective chromatic number of $G$ is the chromatic number of the two-step graph of $G$, which is the graph with the same vertex set as $G$ in which two vertices are adjacent if they have a common neighbor. The concept of injective coloring has been studied by many authors, while in the present paper we approach it from two novel perspectives, related to open packings and the two-step graph operation.  We prove several general bounds on the injective chromatic number expressed in terms of the open packing number. In particular, we prove that
$\chi_{i}(G)\geq \frac{1}{2}+\sqrt{\frac{1}{4}+\frac{2m-n}{\opack}}$ holds
for any connected graph $G$ of order $n\geq2$, size $m$, and the open packing number $\opack$, and characterize the class of graphs attaining the bound. Regarding the well-known bound $\chi_i(G)\ge \Delta(G)$, we describe the family of extremal graphs and prove that deciding when the equality holds (even for regular graphs) is NP-complete, solving an open problem from an earlier paper. Next, we consider the chromatic number of the two-step graph of a graph, and compare it with the clique number and the maximum degree of the graph. We present two large families of graphs in which $\chi_i(G)$ equals the cardinality of a largest clique of the two-step graph of $G$. Finally, we consider classes of graphs that admit an injective coloring in which all color classes are maximal open packings. We give characterizations of three subclasses of these graphs among graphs with diameter $2$, and find a partial characterization of hypercubes with this property.
\end{abstract}

 \vspace{1mm}



\section{Introduction and preliminaries}

Throughout the paper, we consider $G$ as a finite simple graph with vertex set $V(G)$ and edge set $E(G)$.
Recall that a \emph{(vertex) coloring} of $G$ is a labeling of the vertices of $G$ so that any two adjacent vertices have distinct labels. The \emph{chromatic number} of $G$, denoted $\chi(G)$, is the smallest number of labels in a coloring of $G$. For some additional information on coloring problems, we refer the reader to~\cite{JT}.

A function $f:V(G)\rightarrow\{1,\dots,k\}$ is an {\em injective $k$-coloring function} if no vertex $v$ is adjacent to two vertices $u$ and $w$ with $f(u)=f(w)$. For an injective $k$-coloring function $f$, the set of color clases $\big{\{}\{v\in V(G)\mid f(v)=i\}:\,1\leq i\leq k\big{\}}$ is an \emph{injective $k$-coloring} of $G$ (or simply an {\em injective coloring} if $k$ is clear from the context). The minimum $k$ for which a graph $G$ admits an injective $k$-coloring is the {\em injective chromatic number} of $G$, and is denoted by $\chi_{i}(G)$. Injective colorings were introduced in \cite{hkss}, and further studied in~\cite{bu-2009, cky-2010, lst-2009, pp, sy} for   just a few examples.

In this paper, we focus on two approaches to injective colorings of graphs that have seemingly not been much explored. The following concept was mentioned also in the seminal paper of Hahn, Kratochv\'{i}l, \v{S}ir\'{a}\v{n} and Sotteau~\cite{hkss}.  For a given graph $G$, the {\em two-step graph} $\mathcal{N}(G)$ of $G$ is the graph having the same vertex set as $G$ with an edge joining two vertices in $\mathcal{N}(G)$ if and only if they have a common neighbor in $G$. These graphs were introduced in~\cite{av} and investigated later in~\cite{bd,eh,lr}, while in~\cite{hkss} the concept was referred to as the common neighbor graph of a graph $G$.  Taking into account the fact that a vertex subset $S$ is independent in $\mathcal{N}(G)$ if and only if every two vertices of $S$ have no common neighbor in $G$, we can readily observe that
\begin{equation}\label{open-previous}
\chi_{i}(G)=\chi\big{(}\mathcal{N}(G)\big{)}.
\end{equation}

The second approach to injective colorings is motivated by a concept that is close to graph domination studies. Namely, a subset $B\subseteq V(G)$ is an {\em open packing} 
in a graph $G$ if for any two distinct vertices $u,v\in B$, we have $N_{G}(u)\cap N_{G}(v)=\emptyset$. The {\em open packing number} of $G$, denoted by $\opack(G)$, is the maximum cardinality among all open packings in $G$. See~\cite{hs}, where the concept of open packing was introduced, and~\cite{Rall}, where it was efficiently used in the study of total domination in direct products of graphs. It is easy to see that every color class of an injective coloring of a graph $G$ is an open packing in $G$, and consequently, the injective chromatic number is the minimum number of classes in a partition of $G$ into open packings. We give more details about this connection in the next subsection, where we also present several other well studied concepts that are related to injective colorings of graphs.

\subsection{Related concepts}

In this subsection, we thus present some strong relationships between the concepts mentioned above and other ones related to some classical topics of colorings and domination in graphs.

We start with distance coloring of graphs, which was initiated by Kramer and Kramer~ \cite{kk2, kk1} in $1969$. A {\em $2$-distance coloring} of a graph $G$ is a mapping of $V(G)$ to a set of colors such that any two vertices at distance at most two receive different colors (that is, every vertices that are adjacent or have a common neighbor are colored differently). The minimum number of colors $k$ for which there is a $2$-distance coloring of $G$ is the {\em $2$-distance chromatic number}, $\chi_{2}(G)$, of $G$. Clearly, a $2$-distance coloring generates a partition of $V(G)$ into sets of vertices having the same color.

The {\em open neighborhood} of a vertex $v$ is denoted by $N_{G}(v)$, and its {\em closed neighborhood} is $N_{G}[v]=N_{G}(v)\cup \{v\}$. The {\em degree} of a vertex $u$ of $G$ is $\deg_{G}(u)=|N_{G}(v)|$, while the {\em minimum} and {\em maximum degrees} of $G$ are denoted by $\delta(G)$ and $\Delta(G)$, respectively. Also, a subset $S\subseteq V(G)$ is a {\em dominating set} in $G$ if each vertex in $V(G)\backslash S$ has at least one neighbor in $S$. The {\em domination number} $\gamma(G)$ is the minimum cardinality among all dominating sets in $G$. If $G$ has no isolated vertices, then a subset $S\subseteq V(G)$ is a {\em total dominating set} of $G$ if each vertex in $V(G)$ has a neighbor in $S$, and the smallest cardinality of a total dominating set in $G$ is the {\em total domination number}, $\gamma_t(G)$, of $G$. It is well-known that $\opack(G)\leq \gamma_{t}(G)$ for every graph $G$ with no isolated vertices; see~\cite{Rall}.

For more information on domination theory, the reader can consult \cite{hhh,hhs}. Similarly as coloring can be regarded as the partition into independent sets, in domination theory a partition into dominating sets of different types has also been widely considered.
The study of the corresponding parameter, the {\em domatic number} $d(G)$ of a graph $G$, was initiated by Cockayne and Hedetniemi~\cite{ch}. A closely related topic to domination is that of $2$-packing, and thus a graph partitioning problems extended to $2$-packings promises an interesting parameter.

In this concern, a subset $B\subseteq V(G)$ is a {\em $2$-packing} (or simply a {\em packing}) in $G$ if for every pair of distinct vertices $u,v\in B$, we have $N_{G}[u]\cap N_{G}[v]=\emptyset$.
The {\em packing number} $\rho(G)$ is the maximum cardinality among all packings in $G$. As announced above, one may be interested in a partition of the vertex set of a graph  $G$ into $2$-packings. A partition $\mathbb{P}=\{P_{1},\dots,P_{|\mathbb{P}|}\}$ of $V(G)$ is a {\em packing partition} 
if $P_{i}$ is a packing in $G$ for each $i$, $1\leq i\leq|\mathbb{P}|$. The {\em packing partition number} $p(G)$ is the minimum cardinality of a packing partition of $G$.

In contrast with the construction of the two-step graph $\mathcal{N}(G)$ of a graph $G$ mentioned earlier, we recall the following concept. Given a graph $G$, the {\em closed neighborhood graph}, $\mathcal{N}_{c}(G)$, of $G$, has vertex set $V(G)$, and two distinct vertices $u$ and $v$ are adjacent in $\mathcal{N}_{c}(G)$ if and only if $N_{G}[u]\cap N_{G}[v]\neq \emptyset$. See~\cite{bkr} for a recent use of this concept, and note that $\mathcal{N}_{c}(G)$ is also well known under the name the {\em square}, $G^{2}$, of $G$.
With this construction in mind, we observe that a $2$-distance coloring of a graph $G$ is essentially the same as a coloring of its closed neighborhood graph. Furthermore, it is easily seen that the $2$-distance coloring problem is equivalent to the problem of partitioning the vertex set of a graph into packings. Altogether,  we observe the following equalities:
\begin{equation}\label{closed}
\chi_{2}(G)=\chi(G^{2})=\chi\big{(}\mathcal{N}_{c}(G)\big{)}=p(G).
\end{equation}


Motivated by the existence of packing partitions and open packings, we can say that a partition $\mathbb{P}=\{P_{1},\dots,P_{|\mathbb{P}|}\}$ of the vertex set of a graph $G$ is an {\em open packing partition} 
if $P_{i}$ is an open packing in $G$ for each $1\leq i\leq|\mathbb{P}|$. The {\em open packing partition number} $p_{o}(G)$ is the minimum cardinality among all open packing partitions of $G$. 
It turns out that such partitions can be understood from other perspectives, since we can readily see that for any $k$-injective coloring function $f$, the set of color classes $\big{\{}\{v\in V(G)\mid f(v)=i\}:\,1\leq i\leq k\big{\}}$ forms an open packing partition of $G$ and vice versa. This fact, together with~\eqref{open-previous} and the definition of two-step graphs, leads to
\begin{equation}\label{open}
\chi_{i}(G)=\chi\big{(}\mathcal{N}(G)\big{)}=p_{o}(G),
\end{equation}
which is an open neighborhood analogue of (\ref{closed}). This establishes another relationship between coloring and domination theories in graphs.

Another close relative of two-step graphs and injective colorings are exact distance-$2$ graphs and their colorings. Given a graph $G$, the {\em exact distance-$p$ graph} $G^{[\natural p]}$ has $V(G)$ as its vertex set, and two vertices are adjacent whenever the distance between them in $G$ equals $p$. The concept  was introduced back in 1983 by Simi\'{c}~\cite{Sim83} who investigated when the exact distance-$p$ graphs coincide with line graphs.  A different kind of approach to exact distance graphs was given in a series of papers~\cite{DVORAK, PAYAN, PJWAN, ZIEGLER} where the so-called cube-like graphs were considered, which arise as exact distance graphs of hypercubes. Only a decade ago, the concept was rediscovered by Ne\v set\v ril and Ossona de Mendez, which initiated further studies. The main object of many of the mentioned investigation is the chromatic number of exact distance graphs. In particular, a problem from~\cite{NESET2}, attributed to van den Heuvel and Naserasr,  was asking about the boundedness of the chromatic number of $G\pow{p}$ when $G$ is a planar graph and $p$ is an odd integer, while in~\cite{BOUSQUET} the question was answered in the negative by using a class of trees. In some recent papers, exact distance-$p$ colorings were connected with generalized colorings~\cite{HEUVEL}, they were studied in subcubic planar graphs~\cite{FOUCAUD}, and in various graph products~\cite{bgkt}. It was noted in~\cite{FOUCAUD} that for any graph $G$,
$$\chi(G^{[\natural 2]})\le \chi_i(G)\le \chi_{2}(G),$$
and if $G$ is triangle-free, then $\chi(G^{[\natural 2]})= \chi_i(G)$. Actually, this is a direct consequence of the fact that the two-step graph $\mathcal{N}(G)$ of a triangle-free graph $G$ coincides with $G^{[\natural 2]}$. Moreover, observe that $G^{[\natural 2]}$ is always a spanning subgraph of $\mathcal{N}(G)$, where $G^{[\natural 2]}$ can be obtained from $\mathcal{N}(G)$ by removing the edges that belong to a triangle in $G$.

\subsection{Other terminologies, notations and plan of the article}

Given subsets $A,B\subseteq V(G)$, let $[A,B]$ denote the set of all edges with one end-vertex in $A$ and the other in $B$. Given a set $S\subseteq V(G)$, by $G[S]$ we denote the subgraph of $G$ induced by $S$. The maximum cardinality of a set $S$ in $G$ such that $G[S]$ is a complete graph is the {\em clique number} of $G$, denoted $\omega(G)$.
A graph $G$ is {\em chordal} if it contains  no induced cycle of length greater than $3$. It is well known that chordal graphs $G$ are perfect; that is, for any induced subgraph $H$ of $G$, we have $\omega(H)=\chi(H)$.  By a $\chi_{i}(G)$-coloring, a $\chi_{2}(G)$-coloring and a $\opack(G)$-set we mean an injective coloring, a $2$-distance coloring and an open packing of $G$ of cardinality $\chi_{i}(G)$, $\chi_{2}(G)$ and $\opack(G)$, respectively.

For the following two standard products of graphs $G$ and $H$ (see~\cite{ImKl}), the vertex set of the product is $V(G)\times V(H)$. In the edge set of the \emph{Cartesian product} $G\Box H$ two vertices are adjacent if they are adjacent in one coordinate and equal in the other. On the other hand, in the edge set of the \emph{direct product} $G\times H$ two vertices are adjacent if they are adjacent in both coordinates. 
Note that these two products are associative and commutative~\cite{ImKl}. We use the book of West~\cite{we} as a reference for graph theory terminology and notation which are not explicitly defined here.

In this paper, we continue the study of the injective chromatic number, or open packing partition number, or equivalently, the chromatic number of two-step graphs from different perspectives. The three terminologies will be used in concordance with the situation in which they appear.
In the following section we concentrate on some general bounds on the injective chromatic number of graphs. Noting the general lower bound $\chi_i(G)\ge \Delta(G)$, we prove that for any $r$-regular graph $G$ it is NP-complete to decide whether $\chi_i(G)=r$. In contrast to that, we give a structural characterization of graphs $G$ with $\chi_i(G)=\Delta(G)$, which has a particularly nice form in regular graphs, yet, it does not lead to an efficient algorithm. We also present some bounds on the injective chromatic number of graphs expressed in terms of their open packing number, and also a sharp lower bound expressed in terms of the order, size and the open packing number of a graph. In Section~\ref{Sect:final}, we study graphs $G$ with the property that $\chi_i(G)=\omega(\big{(}\mathcal{N}(G)\big{)}$. We prove that $\omega(\big{(}\mathcal{N}(G)\big{)}$ can exceed $\Delta(G)$ for an arbitrarily large amount. On the other hand, we present two large families of graphs $G$ with $\chi_i(G)=\omega(\big{(}\mathcal{N}(G)\big{)}$, notably, the graphs with no induced even cycles, and the graphs whose complements are bipartite. Finally, in Section~\ref{Sect:perfect}, we consider the graphs that admit injective colorings such that all color classes are maximal open packings; in this study, three classes naturally appear depending on how strict conditions we impose. We characterize all three classes of graphs among graphs with diameter $2$ and among even cycles, respectively, and give some initial related results in the class of hypercubes.




\section{General bounds on the injective chromatic number}\label{Sect:bounds}


Let $u\in V(G)$ be a vertex of maximum degree and let $\mathbb{B}=\{B_{1},\dots,B_{\chi_{i}(G)}\}$ be a $\chi_{i}(G)$-coloring. Since $B_{j}$ is an open packing in $G$ for each $1\leq j\leq \Delta(G)$, $u$ has at most one neighbor in $B_{j}$ and hence $\sum_{j=1}^{\chi_{i}(G)}|N(u)\cap B_{j}|\leq \chi_{i}(G)$. So,
\begin{equation}\label{EQ1}
\chi_{i}(G)\geq \Delta(G).
\end{equation}
This simple but important inequality will turn out to be useful in some places in this paper.

A complete characterization of graphs $G$ for which $\chi_{i}(G)=\Delta(G)$ was mentioned as an open problem by Panda and Priyamvada \cite{pp}. In this section, we discuss the complexity and structural aspects of this problem, leading to a complete solution to it. Let us mention that the decision version of the injective chromatic number was shown to be NP-complete even when restricted to (certain subclasses of) bipartite graphs~\cite{jin-2013}.

In what follows, we prove that it is NP-complete to decide whether $\chi_{i}(G)=\Delta(G)$ for a given graph $G$. Indeed, we give a stronger result by limiting on the case when $G$ is regular. For this purpose, we make use of the following well known result due to Leven and Galil \cite{lg} regarding the edge chromatic number $\chi'(G)$ of a graph $G$, which indeed represents the chromatic number of the line graph of $G$.

\begin{lemma}\emph{(\cite{lg})}\label{LG}
For any fixed $r$, the problem of deciding whether $\chi'(G)$ equals $r$ or $r+1$ for an $r$-regular graph $G$ is NP-complete.
\end{lemma}

\begin{theorem}\label{thm:regular}
Given an $r$-regular graph $G$ with $r\geq 3$, it is NP-complete to decide whether $\chi_{i}(G)=r$.
\end{theorem}
\begin{proof}
Let $G$ be an $r$-regular graph with $r\ge 3$. Based on \eqref{EQ1}, it is clear that the problem of deciding whether $\chi_{i}(G)=r$ is in NP. To see the NP-completeness of the problem, we establish a reduction from the problem given in Lemma~\ref{LG}. Let $G'$ be obtained from $G$ by replacing each $uv\in E(G)$ by a $4$-path $u(u,v)(v,u)v$ and creating an $r$-clique on the set of vertices $\{(v,w)\mid w\in N(v)\}$ for each $v\in V(G)$. Now let $G''=G'-V(G)$. (The structure of $G''$, by a different expression, was introduced in~\cite{ht}.) It is routine by the construction that $G''$ is an $r$-regular graph as well. Moreover, we observe that every vertex of $G''$ belongs to a clique of cardinality $r$.

Suppose that $\chi'(G)=r$ and that $f$ is an $r$-edge coloring of $G$. We define $f'$ on $V(G'')$ by $f'\big{(}(u,v)\big{)}=f'\big{(}(v,u)\big{)}=f(uv)$ for each edge $uv$ of $G$. Suppose to the contrary that there exists a vertex $(u,v)$ adjacent to distinct vertices $(u',v')$ and $(u'',v'')$ in $G''$ with $f'\big{(}(u',v')\big{)}=f'\big{(}(u'',v'')\big{)}$. In particular, we have $f(u'v')=f(u''v'')$ by the definition. Since $f$ is an $r$-edge coloring of $G$, it follows that the edges $u'v',u''v''\in E(G)$ have no shared endpoints. This implies, by the structure of $G''$, that there are two $r$-cliques $A$ and $B$ in $G''$ such that $(u',v')\in A$ and $(u'',v'')\in B$. Since each vertex in an $r$-clique $K$ has precisely one neighbor outside $K$, it follows that $(u,v)\in A\cup B$. We therefore assume, without loss of generality, that $(u,v)\in A$. With this in mind and taking the structure of $G''$ into account, we may assume that $u=u'$. This necessarily implies that $v\neq v'$. Therefore, $uv$ and $uv'$ are two distinct edges in $G$. On the other hand, because $(u,v)(u'',v'')\in E(G'')$, it follows from the structure of $G''$ that $(u'',v'')=(v,u)$. Therefore, the edges $u'v'=uv'$ and $u''v''=uv$ in $G$ have the vertex $u$ in common. This contradicts the fact that $f(u'v')=f'\big{(}(u',v')\big{)}=f'\big{(}(u'',v'')\big{)}=f(u''v'')$. Therefore, $f'$ is an injective coloring of $G''$ with $r$ colors. So, $\chi_{i}(G'')\leq r$. This results in the equality in view of the inequality (\ref{EQ1}).

Conversely, let $\chi_{i}(G'')=r$. Let $g$ be a $\chi_{i}(G'')$-coloring. Note that each vertex $x$ of $G$ turns into a unique $r$-clique $A_{x}$ in $G''$. Since $g$ is an injective coloring of $G''$ and because $r\geq3$, $g$ assigns $r$ distinct colors to the vertices in $A_{x}\subseteq V(G'')$ for each $x\in V(G)$. Notice that each vertex $(x,y)\in A_{x}$ necessarily has its $r$th neighbor $(y,x)\in A_{y}$, for some $y\in N_{G}(x)$, with the color $g\big{(}(x,y)\big{)}=g\big{(}(y,x)\big{)}$. We now define $h$ on $E(G)$ by $h(xy)=g\big{(}(x,y)\big{)}=g\big{(}(y,x)\big{)}$. Let $xy$ and $yz$ be two edges of $G$. We observe that $(y,x),(y,z)\in A_{y}$ in the graph $G''$. Since $g$ is an injective coloring of $G''$ and because $|A_{y}|=r\geq3$, it follows that $g$ assigns different colors to $(y,x)$ and $(y,z)$. Therefore, $h(xy)=g\big{(}(y,x)\big{)}\neq g\big{(}(y,z)\big{)}=h(yz)$. Hence, $h$ is an $r$-edge coloring of $G$. Therefore, $\chi'(G)\leq r$. This leads to $\chi'(G)=r$ due to the fact that $\chi'(F)\in\{\Delta(F),\Delta(F)+1\}$ for each simple graph $F$.

In fact, we have proved that $\chi'(G)=r$ if and only if $\chi_{i}(G'')=r$. In view of this, Lemma \ref{LG} completes the proof.
\end{proof}

Next,  we give a structural characterization of all graphs $G$ for which $\chi_{i}(G)=\Delta(G)$. In particular, this results in a simple characterization of all $r$-regular graph $G$ for which $\chi_{i}(G)=r$ (see Corollary \ref{cor:extremal}), which we present next.

Let $\Theta$ be the  family of graphs $G$ defined as follows. We begin with any graphs $H_{1},\dots,H_{t}$ of maximum degree at most $1$, so that at least one of them, say $H_{k}$, is nontrivial (that is, it contains at least one edge). Let $v$ be a non-isolated vertex of $H_k$. Let $G$ be obtained from the disjoint union $H_{1}+\cdots+H_{t}$ by
\begin{itemize}
  \item[$(a)$] joining $v$ by an edge to precisely one vertex of $H_{i}$ for each $1\leq i\neq k\leq t$, and
  \item[$(b)$] by adding some edges with one end-vertex in $V(H_{i})$ and the other one in $V(H_{j})$ such that $[V(H_{i}),V(H_{j})]$ is a matching, for each $1\leq i\neq j\leq t$ (note that $[V(H_{i}),V(H_{j})]=\emptyset$ can also be taken as a matching).
\end{itemize}
Note that, for instance, the path $P_4$ and the cycle $C_4$ are graphs of the family $\Theta$.

\begin{theorem}\label{Extremal}
If $G$ is an arbitrary graph, then $\chi_{i}(G)=\Delta(G)$ if and only if $G\in \Theta$.
\end{theorem}
\begin{proof}
Suppose first that $G\in \Theta$. It is clear from the construction that $\mathbb{H}=\{V(H_1),\dots,V(H_t)\}$ forms an injective coloring of $G$. Moreover, $v$ is a vertex of $G$ with $\deg(v)=t=\Delta(G)$. Therefore, $\chi_{i}(G)\leq|\mathbb{H}|=t=\Delta(G)$. So, $\chi_{i}(G)=\Delta(G)$ by (\ref{EQ1}).

Conversely, suppose that $\chi_{i}(G)=\Delta(G)$. Let $\{B_{1},\dots,B_{\Delta(G)}\}$ be a $\chi_{i}(G)$-coloring of $G$. For every $1\leq i,j\leq \Delta(G)$, every vertex in $B_{i}$ has at most one neighbor in $B_{j}$ because both $B_{i}$ and $B_{j}$ are open packings. This shows that ($i$) $\Delta(G[B_i])\leq1$ for each $1\leq i\leq \Delta(G)$, and ($ii$) $[B_{i},B_{j}]$ is a matching in $G$ for each $1\leq i\neq j\leq \Delta(G)$.

Now let $v\in B_{j}$ be a vertex of maximum degree. Since $|N(v)\cap B_i|\leq1$ for every $1\leq i\leq \Delta(G)$ and that $\deg(v)=\Delta(G)$, it follows that $v$ has precisely one neighbor in $B_i$ for each $1\leq i\leq \Delta(G)$.
It is now readily seen that $v,j,\Delta(G)$ and $G[B_1],\dots,G[B_{\Delta(G)}]$ have the same roles as $v,k,t$ and $H_{1},\dots,H_{t}$, respectively, have in the definition of the family $\Theta$. Thus, $G\in \Theta$.
\end{proof}

Hahn et al.~\cite{hkss} characterized the extremal $r$-regular graphs for the lower bound given in (\ref{EQ1}). Their characterization is based on some algebraic and topological techniques. However, the family $\Theta$ would be much clearer when restricted to $r$-regular graphs. In fact, we have the following immediate consequence of Theorem \ref{Extremal}.

\begin{corollary}\label{cor:extremal}
If $G$ is an $r$-regular graph, then $\chi_{i}(G)=r$ if and only if $G$ is obtained from an $r$-partite graph $H$ with partite sets $H_{1},\cdots,H_{r}$, such that\\
$(1)$ $|H_{1}|=\cdots=|H_{r}|\cong0$ \emph{(}mod $2$\emph{)} and\\
$(2)$ $[H_{i},H_{j}]$ is a perfect matching for each $1\leq i\neq j\leq r$,\\
 by making a perfect matching using the vertices in $H_{i}$ for each $1\leq i\leq r$.
\end{corollary}

The following sharp lower and upper bound on the injective chromatic number in terms of the $2$-distance number of a graph were found in~\cite{ko}: $$\frac{\chi_{2}(G)}{2}\leq \chi_{i}(G)\leq \chi_{2}(G).$$
Next, we present a sharp lower and upper bound on the injective chromatic number in terms of the open packing number of a graph.

It is shown in \cite{ss} that for any graph $G$ on at least three vertices, $\opack(G)=1$ if and only if $\diam(G)\leq2$ and every edge of $G$ lies on a triangle. It is readily observed that this is also a necessary and sufficient condition for $\chi_{i}(G)=|V(G)|$.

\begin{proposition}\label{propo}
If $G$ is a graph of order $n$, then
$$\frac{n}{\opack(G)}\leq \chi_{i}(G)\leq n-\opack(G)+1$$
and the bounds are sharp.
\end{proposition}
\begin{proof}
$(i)$ Let $\mathbb{B}=\{B_1,\dots,B_{|\mathbb{B}|}\}$ be a $\chi_{i}(G)$-coloring. Since every $B_i$ is an open packing in $G$, we have $n=\sum_{1\leq i\leq|\mathbb{B}|}|B_i|\leq|\mathbb{B}|\opack(G)=\chi_{i}(G)\opack(G)$. Hence, $\chi_{i}(G)\geq n/\opack(G)$. On the other hand, if $B$ is a $\opack(G)$-set, then $\{B\}\cup \big{\{}\{g\}\mid g\in V(G)\setminus B\big{\}}$ is an injective coloring of $G$ of cardinality $n-\opack(G)+1$. So, $\chi_{i}(G)\leq n-\opack(G)+1$.

That the lower bound is sharp can be seen by considering the cycles $C_{4m}$ or paths $P_{4m}$ for which $\chi_{i}(C_{4m})=\chi_{i}(P_{4m})=2=4m/\opack(C_{4m})=4m/\opack(P_{4m})$ since $\opack(C_{4m})=\opack(P_{4m})=2m$. The upper bound is sharp for $K_{n}$ as well as for $K_{1,n-1}$ on $n\geq3$ vertices.
\end{proof}

In the next result, we give a lower bound on the injective chromatic number of a graph $G$ in terms of its order, size and open packing number. Despite the fact that it is not comparable with $\Delta(G)$, the family of extremal graphs for the two lower bounds are the same as the family $\Theta$ when it is restricted to regular graphs. For the sake of convenience, we let $\Theta_r$ be the family of all regular graphs in $\Theta$. Note that the members of $\Theta_r$ are represented in the statement of Corollary~\ref{cor:extremal}.

\begin{theorem}\label{Delta}
For any connected graph $G$ of order $n\geq2$ and size $m$,
$$\chi_{i}(G)\geq \frac{1}{2}+\sqrt{\frac{1}{4}+\frac{2m-n}{\opack(G)}}.$$
The equality holds if and only if $G\in \Theta_r$.
\end{theorem}
\begin{proof}
Let $\mathcal{P}=\{P_{1},\dots,P_{\chi_{i}(G)}\}$ be a $\chi_{i}(G)$-coloring. Relabeling the subscripts if necessary, we may assume that $|P_{1}|\leq\cdots\leq|P_{\chi_{i}(G)}|$. By definition, every vertex in $P_{i}$ has at most one neighbor in $P_{j}$ for all $1\leq i<j\leq \chi_{i}(G)$. Moreover, the subgraph of $G$ induced by $P_{i}$, for each $1\leq i\leq \chi_{i}(G)$, has at most $|P_{i}|/2$ edges. We therefore conclude that,
\begin{equation}\label{Inequ}
\begin{array}{lcl}
m=\sum_{1\leq i<j\leq \chi_{i}(G)}|[P_{i},P_{j}]|+\sum_{1\leq i\leq \chi_{i}(G)}|[P_{i},P_{i}]|&\leq& \sum_{i=1}^{\chi_{i}(G)-1}|P_{i}|(\chi_{i}(G)-i)+\dfrac{n}{2}\vspace{0.5mm}\\
&\leq& |P_{\chi_{i}(G)}|\sum_{i=1}^{\chi_{i}(G)-1}(\chi_{i}(G)-i)+\dfrac{n}{2}\vspace{0.5mm}\\
&\leq&\big{(}\dfrac{\chi_{i}(G)(\chi_{i}(G)-1)}{2}\big{)}\opack(G)+\dfrac{n}{2}.
\end{array}
\end{equation}
Solving the inequality chain (\ref{Inequ}) for $\chi_{i}(G)$, we get $\chi_{i}(G)\geq\big{(}1+\sqrt{1+4(2m-n)/\opack(G)}\big{)}/2$.

Suppose that the lower bound holds with equality. Therefore, all three inequalities in (\ref{Inequ}) necessarily hold with equality. In particular, together the second and third resulting equalities show that $|P_{1}|=\cdots=|P_{\chi_{i}(G)}|=\opack(G)$. Moreover, the first one implies that any vertex in $P_{i}$ has precisely one neighbor in each other $P_{j}$, and $|[P_{i},P_{i}]|=|P_{i}|/2$ for all $1\leq i\leq \chi_{i}(G)$. This shows that the edges of each subgraph $G[P_{i}]$ form a perfect matching. It is now easy to observe that the resulting subgraph of $G$ by removing $\cup_{i=i}^{\chi_{i}(G)}[P_{i},P_{i}]$ is a $\chi_{i}(G)$-partite graph with partite sets $P_i$ of cardinality $|P_i|\cong0$ (mod $2$) for $1\leq i\leq \chi_{i}(G)$. That $G\in \Theta_r$ can be seen by observing the fact that $\chi_{i}(G)$ and $P_{1},\cdots,P_{\chi_{i}(G)}$ have the same roles as $r$ and $H_{1},\cdots,H_{r}$ have in the description of the members of $\Theta_r$, respectively.

Conversely, let $G\in \Theta_r$. By the structure of $G$, the vertex partition $\mathbb{X}=\{H_{1},\dots,H_{r}\}$ is an injective coloring of $G$. Hence, $\chi_{i}(G)\leq r$. Since $H_{1}$ is both an open packing and a total dominating set in $G$, it follows that $|H_{1}|\leq \opack(G)\leq \gamma_{t}(G)\leq|H_{1}|$, and so $\opack(G)=|H_{1}|=n/r$. Moreover, we have $r=2m/n$ as $2m=\sum_{v\in V(G)}\deg_{G}(v)=rn$. A simple calculation then shows that $\big{(}1+\sqrt{1+4(2m-n)/\opack(G)}\big{)}/2=2m/n=r\geq \chi_{i}(G)$. This implies the equality in the lower bound.
\end{proof}

The following theorem shows that the simple lower bound given in (\ref{EQ1}) gives the exact values of $\chi_{i}$ when dealing with nontrivial trees. An additional goal is to give a proof for it that can be implemented as a polynomial-time algorithm for obtaining an optimal injective coloring of a tree.

We recall that the {\em eccentricity} of a vertex $v$ in a graph $G$, written $\varepsilon_{G}(v)$, is the maximum of distances from $v$ to other vertices of $G$.

\begin{theorem}\label{Tree}
For any tree $T$ on at least two vertices, $\chi_{i}(T)=\Delta(T)$.
\end{theorem}
\begin{proof}
We have $\chi_{i}(T)\geq \Delta(T)$ by the inequality (\ref{EQ1}). Therefore, it suffices to construct an injective coloring of $T$ of cardinality $\Delta(T)$.

Let $r$ be a vertex of maximum degree in $T$. We root $T$ at $r$. Then, any vertex at distance $\varepsilon_{T}(r)$ from $r$ is a leaf. We assign $1$ to $r$ and the colors $1,\dots,\Delta(T)$ to the children of $r$ so that any of them takes a unique color. If $T$ is a star, then we are done. So, let $v$ be a child of $r$ colored with $i\in \{1,\dots,\Delta(T)\}$ which is not a leaf. Since $\deg(v)\leq \deg(r)=\Delta(T)$, it follows that $v$ has at most $\Delta(T)-1$ children. We now assign $\deg(v)-1$ colors among $\{1,\dots,\Delta(T)\}\setminus\{1\}$ to its children so that any of these $\deg(v)-1$ colors appears on only one such a child. This process is continued until all descendants of $v$ are assigned colors among $\{1,\dots,\Delta(T)\}$. Iterating this process for any other none-leaf child of $r$ (if any), any vertex of $T$ takes a color from $\{1,\dots,\Delta(T)\}$ so that no vertex is adjacent to two vertices having the same colors. Therefore, the subsets $V_{i}=\{v\in V(T)\mid v \ \mbox{is colored with}\ i\}$ for $1\leq i\leq \Delta(T)$ give an injective coloring of $T$ of cardinality $\Delta(T)$. This completes the proof.
\end{proof}


\section{Two-step graphs and the equality $\chi_{i}(G)=\omega\big{(}\mathcal{N}(G)\big{)}$}\label{Sect:final}


Based on the relationship \eqref{open}, it is likely that constructing the two-step graphs of some families of graphs can be useful in several situations, in particular in the context of injective colorings. It can be easily observed that if $G$ is a cycle of order $n\ge 3$, then $\mathcal{N}(G)$ is either a cycle of order $n$ (when $n$ is odd) or the disjoint union of two cycles of order $n/2$ (when $n$ is even). Also, if $G$ is the path $P_n$, then $\mathcal{N}(G)$ is formed by two disjoint paths of order $\left\lceil n/2\right\rceil$ and $\left\lfloor n/2\right\rfloor$; if $G$ is the complete graph of order $n\neq2$, then $\mathcal{N}(G)=G=K_n$; and if $G$ is a complete bipartite graph $K_{r,t}$, then $\mathcal{N}(G)$ is formed by two disjoint complete graphs $K_r$ and $K_t$. Moreover, if $G$ is a triangle-free graph of diameter 2, then one can readily observe that $\mathcal{N}(G)$ is isomorphic to the complement graph $\overline{G}$. Notice that this last comment together with equality \eqref{open} leads to $\chi_i(G)=\chi(\overline{G})$ for any triangle-free graph of diameter 2.

On the other hand, by the structure of the two-step graph $\mathcal{N}(G)$ of a graph $G$, we observe that $\omega\big{(}\mathcal{N}(G)\big{)}$ equals the maximum number of vertices of $G$ such that any two of them have a common neighbor. This fact together with relationship \eqref{open} gives sense to considering whether $\chi_{i}(G)=\omega\big{(}\mathcal{N}(G)\big{)}$, since it is well known that $\chi(G)\ge \omega(G)$ for any graph $G$. We can easily observe that $\omega\big{(}\mathcal{N}(G)\big{)}\ge \Delta(G)$. However, it can happen that $\omega\big{(}\mathcal{N}(G)\big{)}\gg\Delta(G)$.
To see this, we consider for instance the direct product graph $K_r\times K_t$ with $r,t\geq3$. By the adjacency rules of direct product and two-step graphs, we observe that $\mathcal{N}(K_r\times K_t)\cong K_{rt}$. Therefore, we have $\omega\big{(}\mathcal{N}(K_r\times K_t)\big{)}=rt$ while $\Delta((K_r\times K_t))=(r-1)(t-1)$. In fact, we have the following result.

\begin{corollary}
\label{cor-omegaDelta}
For any positive integer $\ell$ there exists a graph $G$ such that
$\omega\big{(}\mathcal{N}(G)\big{)}-\Delta(G)>\ell$.
\end{corollary}

In the next result, we give a large family of graphs whose two-step graphs are perfect.

\begin{theorem}\label{Chordal}
Let $G$ be a $C_{2k}$-free graph for each $k\geq2$. Then, $\chi_{i}(G)=\omega\big{(}\mathcal{N}(G)\big{)}$.
\end{theorem}

\begin{proof}
We first prove that $\mathcal{N}(G)$ is a chordal graph. Suppose to the contrary that $C_{k}:v_{1}v_{2}\cdots v_{k}v_{1}$ is a chordless cycle in $\mathcal{N}(G)$ for some $k\geq4$. Therefore, there exist $u_{1},\dots,u_{k}\in V(G)$ such that $\{v_{1},v_{2}\}\subseteq N(u_{1}),\dots,\{v_{k-1},v_{k}\}\subseteq N(u_{k-1}),\{v_{k},v_{1}\}\subseteq N(u_{k})$. Because $C_{k}$ is chordless, it follows that $\{u_{1},u_{2},\dots,u_{k}\}\cap \{v_{1},v_{2},\dots,v_{k}\}=\emptyset$. If the vertices $u_{i}$ are pairwise distinct, then $v_{1}u_{1}v_{2}\cdots v_{k-1}u_{k-1}v_{k}u_{k}v_{1}$ is a cycle in $G$ on $2k$ vertices, a contradiction. Therefore, $u_{s}=u_{t}$ for some $1\leq s<t\leq k$. If $t=s+1$, then $v_{s}v_{s+2}$ is a chord of $C_{k}$ in $\mathcal{N}(G)$. If $t>s+1$, then $v_{s}v_{t}$ is a chord of $C_{k}$ in $\mathcal{N}(G)$. Each case leads to a contradiction. Therefore, $\mathcal{N}(G)$ is chordal, and so it is a perfect graph. Thus, $\chi\big{(}\mathcal{N}(G)\big{)}=\omega\big{(}\mathcal{N}(G)\big{)}$, and by relationship \eqref{open}, $\chi_{i}(G)=\omega\big{(}\mathcal{N}(G)\big{)}$.
\end{proof}

Let $T$ be a tree and $Q$ be a maximum clique in $\mathcal{N}(T)$. So, any two vertices of $Q$ have a common neighbor in $T$. Since $T$ is a tree, $Q$ is independent in $T$. With this in mind, a unique vertex is adjacent to all vertices of $Q$ in $T$. (Indeed, otherwise the subgraph of $T$ induced by the union of $Q$ and the set of vertices each of which is adjacent to at least two vertices in $Q$ contains a cycle, a contradiction.) This implies that $\chi_{i}(T)=\omega\big{(}\mathcal{N}(T)\big{)}=\Delta(T)$. Consequently, Theorem \ref{Tree} is an immediate consequence of Theorem \ref{Chordal}.  However, as mentioned earlier, the proof of Theorem~\ref{Tree} provides an efficient way to obtain an optimal injective coloring of an arbitrary tree.

Note that the condition of being $C_{2k}$-free for each $k\geq2$ cannot be removed in Theorem \ref{Chordal}. To see this, consider the cycle $C_{4t+2}$ for $t\geq2$. It is easy to see that $\mathcal{N}(C_{4t+2})\cong2C_{2t+1}$, and so $\chi_i(C_{4t+2})=\chi\big{(}\mathcal{N}(C_{4t+2})\big{)}=3\neq2=\omega\big{(}\mathcal{N}(C_{4t+2})\big{)}$. In fact, this example shows that the equality in Theorem~\ref{Chordal} does not necessarily hold even if $G$ is a perfect graph. However, it remains true for several infinite families of graphs containing even cycles as induced subgraphs.

\begin{theorem}\label{Compliment}
If the complement $\overline{G}$ of a graph $G$ is a bipartite graph, then $\chi_{i}(G)=\omega\big{(}\mathcal{N}(G)\big{)}$.
\end{theorem}
\begin{proof}
Let $G$ be a graph of order $n$, and let $X$ and $Y$ be the partite sets of $\overline{G}$ with $|X|\leq|Y|$. Since $\overline{G}$ is bipartite, both $X$ and $Y$ are cliques in $G$. If $\overline{G}$ is a complete bipartite graph, then $G$ is isomorphic to the disjoint union $K_{|X|}+K_{|Y|}$. In such a situation, we have
\begin{equation*}
\mathcal{N}(G)\cong\left \{
\begin{array}{lll}
\overline{K_{n}} & \mbox{if}\ \ |Y|\leq2,\\
\overline{K_{|X|}}+K_{|Y|} & \mbox{if}\ \ |X|\leq2\ \mbox{and}\ |Y|\geq3,\\
K_{|X|}+K_{|Y|} & \mbox{if}\ \ |X|\geq3.
\end{array}
\right.
\end{equation*}
In each case, $\chi_{i}(G)=\chi\big{(}\mathcal{N}(G)\big{)}=\omega\big{(}\mathcal{N}(G)\big{)}\in \{1,|Y|\}$.

So, in what follows we may assume that $\overline{G}$ is not complete bipartite. Therefore, $[X,Y]\subseteq E(G)$ is nonempty. We distinguish two cases depending on $|X|$.\vspace{1mm}

\noindent
\textit{Case 1.} $|X|\geq3$. Suppose first that $[X,Y]=\{x_{1}y_{1},\dots,x_{k}y_{k}\}$ is a matching in $G$. Since $Y$ is a clique in $G$, we deduce by the structure of $\mathcal{N}(G)$ that $x_{i}$ is adjacent to all vertices in $Y\setminus\{y_{i}\}$ in $\mathcal{N}(G)$ for each $1\leq i\leq k$. Moreover, $x_{1}y_{1},\dots,x_{k}y_{k}\notin E\big{(}\mathcal{N}(G)\big{)}$. We also observe that $[X\setminus\{x_{1},\cdots,x_{k}\},Y\setminus\{y_{1},\cdots,y_{k}\}]$ is empty in $\mathcal{N}(G)$. In such a situation, $\omega\big{(}\mathcal{N}(G)\big{)}=|Y|$ because $|Y|\geq|X|$. Let $f$ assign the colors
\begin{itemize}
  \item $1,\dots,k$ to the vertices in $\{x_{1},y_{1}\},\dots,\{x_{k},y_{k}\}$, respectively,
  \item $k+1,\dots,|Y|$ to the other vertices of $Y$ (if any), and
  \item $k+1,\dots,|X|$ to the other vertices of $X$ (if any).
\end{itemize}
It is readily seen that $f$ is a coloring of $\mathcal{N}(G)$ that assigns $|Y|$ colors to the vertices of $V\big{(}\mathcal{N}(G)\big{)}=V(G)$. Therefore, $\chi_{i}(G)=\chi\big{(}\mathcal{N}(G)\big{)}\leq|Y|=\omega\big{(}\mathcal{N}(G)\big{)}$. Since also, $\chi\big{(}\mathcal{N}(G)\big{)}\ge \omega\big{(}\mathcal{N}(G)\big{)}$ we obtain the desired equality.

Suppose now that $[X,Y]$ is not a matching. This shows that $|N_{G}(x)\cap Y|\geq2$ or $|N_{G}(y)\cap X|\geq2$ for some $x\in X$ or $y\in Y$, respectively. Let $X'=\{x\in X\,:\, |N_{G}(x)\cap Y|\geq2\}$ and $Y'=\{y\in Y\,:\, |N_{G}(y)\cap X|\geq2\}$. By the adjacency rule of $\mathcal{N}(G)$, we have $N_{\mathcal{N}(G)}(x)\cap Y=Y$ and $N_{\mathcal{N}(G)}(y)\cap X=X$ for any $x\in X'\cup(N_{G}(Y')\cap X)$ and $y\in Y'\cup(N_{G}(X')\cap Y)$, respectively. It is not hard to see that both $Q_{1}=X\cup Y'\cup(N_{G}(X')\cap Y)$ and $Q_{2}=Y\cup X'\cup(N_{G}(Y')\cap X)$ are cliques in $\mathcal{N}(G)$. Let now $X''=X\setminus\big{(}X'\cup(N_{G}(Y')\cap X)\big{)}$ and $Y''=Y\setminus\big{(}Y'\cup(N_{G}(X')\cap Y)\big{)}$. Note that each vertex $x\in X''$ (resp. $y\in Y''$) has at most one neighbor in $Y$ (resp. $X$), and this neighbor belongs to $Y''$ (resp. $X''$), necessarily. This implies that $[X'',Y'']$ is a matching in $G$. Suppose that $[X'',Y'']=\{x_{1}''y_{1}'',\dots,x_{r}''y_{r}''\}$ (see Figure \ref{Fig2}). Again by using the adjacency rule of $\mathcal{N}(G)$, we observe that none of the edges $x_{1}''y_{1}'',\dots,x_{r}''y_{r}''$ appears in $\mathcal{N}(G)$.

Suppose first that $|X''|\geq|Y''|$. We consider a function $g$ that assigns the colors
\begin{itemize}
  \item $1,\dots,r$ to the vertices in $\{x_{1}'',y_{1}''\},\dots,\{x_{r}'',y_{r}''\}$, respectively,
  \item $r+1,\dots,|X''|$ to the other vertices in $X''$ (if any),
  \item $|X''|+1,\dots,|X|$ to the vertices in $X\setminus X''$,
  \item $|X|+1,\dots,|Q_{1}|$ to the vertices in $Q_{1}\setminus X$, and
  \item $r+1,\dots,|Y''|$ to the other vertices in $Y''$ (if any).
\end{itemize}
Note that $g$ is a coloring of $\mathcal{N}(G)$ assigning $|Q_{1}|$ colors to the vertices of $\mathcal{N}(G)$. Therefore, $\chi_{i}(G)=\chi\big{(}\mathcal{N}(G)\big{)}\leq|Q_{1}|\leq \omega\big{(}\mathcal{N}(G)\big{)}$. Now, since $\chi\big{(}\mathcal{N}(G)\big{)}\ge \omega\big{(}\mathcal{N}(G)\big{)}$ we obtain the desired equality.

In a similar fashion, we deduce that $\chi_{i}(G)=\chi\big{(}\mathcal{N}(G)\big{)}\leq|Q_{2}|\leq \omega\big{(}\mathcal{N}(G)\big{)}$ when $|Y''|\geq|X''|$, and again as $\chi\big{(}\mathcal{N}(G)\big{)}\ge \omega\big{(}\mathcal{N}(G)\big{)}$ we get the desired equality when $|X|\geq3$.\vspace{1mm}

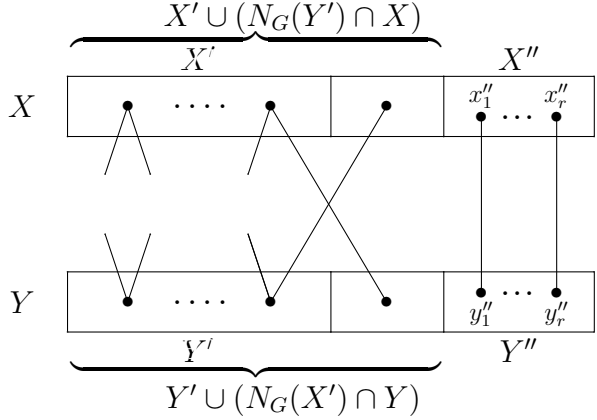
\begin{figure}[ht]
 \centering
\begin{tikzpicture}[scale=.02, transform shape]
\node [draw, shape=circle] (v1) at  (0,0) {};
\node [draw, shape=circle] (v2) at  (350,0) {};
\node [draw, shape=circle] (v3) at  (0,-40) {};
\node [draw, shape=circle] (v4) at  (350,-40) {};
\node [scale=50] at (-30,-20) {\large $X$};

\node [draw, shape=circle] (v1') at  (0,-130) {};
\node [draw, shape=circle] (v2') at  (350,-130) {};
\node [draw, shape=circle] (v3') at  (0,-170) {};
\node [draw, shape=circle] (v4') at  (350,-170) {};
\node [scale=50] at (-30,-150) {\large $Y$};

\draw (v1)--(v2);
\draw (v3)--(v4);
\draw (v1)--(v3);
\draw (v2)--(v4);

\draw (v1')--(v2');
\draw (v3')--(v4');
\draw (v1')--(v3');
\draw (v2')--(v4');

\node [draw, shape=circle] (u1) at  (250,0) {};
\node [draw, shape=circle] (u2) at  (175,0) {};
\node [draw, shape=circle] (u3) at  (250,-40) {};
\node [draw, shape=circle] (u4) at  (175,-40) {};
\node [scale=50] at (87,12) {\large $X'$};
\node [scale=50] at (300,12) {\large $X''$};
\node [scale=50] at (150,39) {\large $X'\cup(N_{G}(Y')\cap X)$};
\node [scale=50] at (150,-215) {\large $Y'\cup(N_{G}(X')\cap Y)$};

\draw (u1)--(u3);
\draw (u2)--(u4);

\node [draw, shape=circle] (w1) at  (250,-130) {};
\node [draw, shape=circle] (w2) at  (175,-130) {};
\node [draw, shape=circle] (w3) at  (250,-170) {};
\node [draw, shape=circle] (w4) at  (175,-170) {};
\node [scale=50] at (87,-182) {\large $Y'$};
\node [scale=50] at (300,-182) {\large $Y''$};

\draw (w1)--(w3);
\draw (w2)--(w4);\vspace{250mm}
\node [scale=50] at (125,18) {$\overbrace{\begin{color}{white}HHHHHHHHHHHHHH\end{color}}$};
\node [scale=50] at (125,-188) {$\underbrace{\begin{color}{white}HHHHHHHHHHHHHH\end{color}}$};

\node [scale=40] at (40,-20) {\large $\bullet$};
\node [draw, shape=circle] (p) at (40,-20) {};
\node [draw, shape=circle] (p1) at (25,-65) {};
\node [draw, shape=circle] (p2) at (55,-65) {};
\draw (p)--(p1);
\draw (p)--(p2);
\node [scale=60] at (73,-20) {\large $.$};
\node [scale=60] at (83,-20) {\large $.$};
\node [scale=60] at (93,-20) {\large $.$};
\node [scale=60] at (103,-20) {\large $.$};
\node [scale=40] at (135,-20) {\large $\bullet$};
\node [draw, shape=circle] (p') at (135,-20) {};
\node [draw, shape=circle] (p1') at (120,-65) {};
\node [draw, shape=circle] (p2') at (212,-150) {};
\draw (p')--(p1');
\draw (p')--(p2');

\node [scale=40] at (212,-20) {\large $\bullet$};

\node [scale=40] at (275,-27) {\large $\bullet$};
\node [scale=40] at (325,-27) {\large $\bullet$};
\node [scale=60] at (291,-27) {\large $.$};
\node [scale=60] at (299,-27) {\large $.$};
\node [scale=60] at (307,-27) {\large $.$};
\node [scale=40] at (275,-13) {\large $x_{1}''$};
\node [scale=40] at (325,-13) {\large $x_{r}''$};

\node [scale=40] at (40,-150) {\large $\bullet$};
\node [draw, shape=circle] (q) at (40,-150) {};
\node [draw, shape=circle] (q1) at (25,-105) {};
\node [draw, shape=circle] (q2) at (55,-105) {};
\draw (q)--(q1);
\draw (q)--(q2);
\node [scale=60] at (73,-150) {\large $.$};
\node [scale=60] at (83,-150) {\large $.$};
\node [scale=60] at (93,-150) {\large $.$};
\node [scale=60] at (103,-150) {\large $.$};
\node [scale=40] at (135,-150) {\large $\bullet$};
\node [draw, shape=circle] (q') at (135,-150) {};
\node [draw, shape=circle] (q1') at (120,-105) {};
\node [draw, shape=circle] (q2') at (212,-20) {};
\draw (q')--(q1');
\draw (q')--(q1');
\draw (q')--(q2');

\node [scale=40] at (212,-150) {\large $\bullet$};

\node [scale=40] at (275,-144) {\large $\bullet$};
\node [scale=40] at (325,-144) {\large $\bullet$};
\node [scale=60] at (291,-144) {\large $.$};
\node [scale=60] at (299,-144) {\large $.$};
\node [scale=60] at (307,-144) {\large $.$};
\node [scale=40] at (275,-158) {\large $y_{1}''$};
\node [scale=40] at (325,-158) {\large $y_{r}''$};

\node [draw, shape=circle] (x_{1}'') at  (275,-27) {};
\node [draw, shape=circle] (y_{1}'') at  (275,-144) {};
\draw (x_{1}'')--(y_{1}'');

\node [draw, shape=circle] (x_{r}'') at  (325,-27) {};
\node [draw, shape=circle] (y_{r}'') at  (325,-144) {};
\draw (x_{r}'')--(y_{r}'');
\end{tikzpicture}
  \caption{The graph $G$ described in Case $1$ of the proof of Theorem \ref{Compliment} when $[X,Y]$ is not a matching.}\label{Fig2}
\end{figure}

\noindent
\textit{Case 2.} $|X|\leq2$. We now consider two possibilities depending on $|Y|$.\vspace{0.75mm}

\noindent
\textit{Subcase 2.1.} $|Y|\geq3$. Since $\overline{G}$ is not complete bipartite, there exists an edge $xy\in E(G)$ in which $x\in X$ and $y\in Y$. If $|X|=1$, then it is easy to see that $\mathcal{N}(G)\cong K_{n}$ or $\mathcal{N}(G)\cong K_{n}-xy$ depending on $|[X,Y]|\geq2$ or $|[X,Y]|=1$, respectively. In both cases, the desired equality holds. Suppose that $x,x'\in X$ are two distinct vertices. If $N_{G}(x)\cap N_{G}(x')\neq \emptyset$, then $\mathcal{N}(G)=K_{n}$. So, we assume that $x$ and $x'$ have no common neighbor in $G$. In such a situation, we conclude that
\begin{equation*}
\mathcal{N}(G)\cong\left \{
\begin{array}{lll}
K_{n}-xx' & \mbox{if both $x$ and $x'$ have at least two neighbors in $Y$};\\
K_{n}-\{xx',x'y'\} & \mbox{if $x$ has at least two neighbors and $x'$ has a unique neighbor}\\
& \mbox{$y'$ in $Y$};\\
H & \mbox{if $x$ has at least two neighbors in $Y$ and $x'$ is not adjacent to}\\
& \mbox{any vertex in $Y$, where $H$ is isomorphic to the graph obtained}\\
& \mbox{from $K_{n-1}$ by joining a new vertex $x'$ to $\deg_{G}(x)-1$ vertices}\\
& \mbox{of $y$};\\
(K_{n-1}-xy)+x'y & \mbox{if $x$ has a unique neighbor $y\in Y$ and $x'$ does not have any}\\
& \mbox{neighbor in $Y$;}\\
F & \mbox{if $x$ and $x'$ have unique neighbors $y$ and $y'$ in $Y$, respectively,}\\
& \mbox{where $F$ is obtained from $K_{n-2}$ on the vertices in $Y$ by joining}\\
& \mbox{new vertices $x$ and $x'$ to all vertices in $Y\setminus\{y\}$ and $Y\setminus\{y'\}$,}\\
& \mbox{respectively;}\\
K_{n-2}+\{x,y\} & \mbox{otherwise}.
\end{array}
\right.
\end{equation*}
In all cases above, we have $\chi_{i}(G)=\chi\big{(}\mathcal{N}(G)\big{)}=\omega\big{(}\mathcal{N}(G)\big{)}$.\vspace{0.75mm}

\noindent
\textit{Subcase 2.2.} $|Y|\leq2$. It is a simple matter to see that
$$\mathcal{N}(G)\in\big{\{}2K_{1},K_{3},K_{2}+K_{1},K_{4},2K_{2},K_{4}-xy\big{\}},$$
in which $x$ and $y$ are two vertices of $K_{4}$. In each case, we get the desired equality. This completes the proof.
\end{proof}

In view of the usefulness of the two-step operation, it is natural to separately consider its structural properties, such as finding forbidden subgraphs in two-step graphs, or characterizing the graphs that can be realized as the two-step graph of some graph. For instance, with respect to the latter, we observe that a triangle-free graph $G$ with $\Delta(G)\ge 3$ is not the two-step graph of any graph.

\begin{remark}
If a connected triangle-free graph $G$ on at least two vertices is the two-step graph of some graph, then $G$ is an odd cycle.
\end{remark}

\begin{proof}
Let $G$ be a connected triangle-free graph of order $n\geq2$ different from an odd cycle. Suppose $G$ is isomorphic to $\mathcal{N}(H)$ for some (necessarily connected) graph $H$. If $H$ has maximum degree $2$, then $H$ is either a cycle or a path, whose two-step graph is either a cycle (when $H$ is an odd cycle) or is disconnected, which is not possible. On the other hand, if $H$ has a vertex $v$ of degree at least $3$, then any three neighbors of $v$ induce a triangle in $\mathcal{N}(H)$, a contradiction.
\end{proof}


\section{Perfect injectively colorable graphs}
\label{Sect:perfect}


A problem that connects coding theory to domination in graphs asks whether there exists a minimum dominating set $D$ in a graph $G$ such that every vertex of $G$ has in its closed neighborhood exactly one vertex of $D$. Such a dominating set is known under different names, such as an {\em efficient dominating set}~\cite{bbs} or an {\em independent perfect dominating set}~\cite{lee} or {\em a perfect code}. The problem of characterizing the graphs that admit such a dominating set goes back to the 1970s~\cite{biggs}, and it is still in general not resolved. Note that any such graph has its domination number equal to the packing number. Furthermore, partitions of the vertex of a graph into perfect codes have been considered as early as in the 1980's, where the main focus was on the class of hypercubes~\cite{phelps}. In relation with these studies and connecting them to our work, we present three types of partitions of graphs into open packings.

First, suppose that a graph $G$ admits a partition into $\opack(G)$-sets. Based on the  relationship~\eqref{open}, this implies that $G$ has an injective coloring in which every color class has the same number of vertices. For instance, if $G$ is a cycle $C_{4r}$, then $\opack(C_{4r})=2r$ and $\chi_{i}(C_{4r})=2$, and the desired partition holds for cycles $C_{4r}$. Other sporadic examples of graphs with this property are the hypercubes $Q_3$ and $Q_4$, the complete bipartite graphs $K_{t,t}$ and the Sierpi\'nski graph $S_4^2$. To formalize, a graph $G$ is a \emph{perfect injectively colorable graph} if it has an injective coloring in which every color class forms a $\opack(G)$-set. Note that such an injective coloring of $G$ is necessarily a $\chi_i(G)$-coloring. Our aim is to initiate the study of perfect injectively colorable graphs, and here we give some initial results. From this definition, and due to \eqref{open}, we readily deduce the following result.

\begin{proposition}
Let $G$ be a graph. Then $G$ is a perfect injectively colorable graph if and only if there is a $\chi(\mathcal{N}(G))$-coloring in which all color classes have the same cardinality equal to $\rho^o(G)$.
\end{proposition}

An application of the result above can be for instance while considering paths $P_n$, $n\ge 2$. As previously mentioned, $\mathcal{N}(P_n)$ is isomorphic to two disjoint paths $P_{\left\lfloor n/2\right\rfloor}$ and $P_{\left\lceil n/2\right\rceil}$. Hence, we deduce the following claims.
\begin{itemize}
  \item If $n=4p$ for some integer $p\ge 1$, then $\mathcal{N}(P_{4p})$ is formed by two disjoint paths $P_{2p}$. This means there exists a $\chi\big{(}\mathcal{N}(P_{4p})\big{)}$-coloring in which every color class has cardinality $2p$. Since $\rho^o(P_{4p})=2p$, it follows $P_{4p}$ is perfect injectively colorable.
  \item If $n=4p+1$ or $n=4p+3$ for some integer $p\ge 1$, then $\mathcal{N}(P_n)$ is formed by two disjoint paths $P_{\left\lfloor n/2\right\rfloor}$ and $P_{\left\lceil n/2\right\rceil}$ (of different orders, one even and one odd). We easily observe that there is a color class in any $\chi\big{(}\mathcal{N}(P_{n})\big{)}$-coloring that has cardinality less than $\rho^o(P_{n})$. Thus, $P_{n}$ is not perfect injectively colorable.
  \item If $n=4p+2$ for some integer $p\ge 1$, then $\mathcal{N}(P_n)$ is formed by two disjoint paths $P_{2p+1}$. Now there is a color class in any $\chi\big{(}\mathcal{N}(P_{n})\big{)}$-coloring that has cardinality less than $\rho^o(P_{n})$. Again, $P_{n}$ is not perfect injectively colorable.
\end{itemize}

We next turn our attention to graphs of diameter $2$. Note that $\opack(G)\le 2$ if $\diam(G)=2$.
If, in addition, $G$ is triangle-free, then a partition of $V(G)$ into $\opack(G)$-sets coincides with the existence of a perfect matching in $G$. We generalize this observation to a characterization of diameter $2$ perfect injectively colorable graphs.

\begin{proposition}\label{prp:PICdiam2}
If $G$ is a graph with $\diam(G)=2$, then $G$ is a perfect injectively colorable graph if and only if either each edge of $G$ lies in a triangle or there exists a perfect matching $M$ in $G$ such that each edge of $M$ does not lie in a triangle.
\end{proposition}
\begin{proof}
Let $G$ be a graph with $\diam(G)=2$. Firstly, if each of its edges lies in a triangle, then every two vertices have a common neighbor, so $\opack(G)=1$. We infer that $\chi_{i}(G)=|V(G)|$, and $G$ is a perfect injectively colorable graph. Secondly, suppose that $G$ has a perfect matching $M$ such that each edge of $M$ does not lie in a triangle. Note that $\{u,v\}$ is an open packing if $uv$ is an edge that does not lie in a triangle. Thus, $\opack(G)=2$, and the endvertices of any edge from $M$ form an $\opack(G)$-set. We readily infer that $G$ is a perfect injectively colorable graph.

Conversely, suppose that $G$ is a perfect injectively colorable graph (with $\diam(G)=2$). We may further assume that there is an edge $e$ that does not lie in a triangle. By the same argument as above, we infer that $\opack(G)=2$. Hence $V(G)$ can be partitioned into $\opack(G)$-sets of cardinality $2$. Since $\diam(G)=2$, each $\opack(G)$-set consists of two adjacent vertices. Thus, $G$ has a matching, which saturates all vertices of $G$, that is, a perfect matching.
\end{proof}

A weaker form of perfectness is obtained when one requires that each open packing in a partition given by an injective coloring is maximal, though not necessarily maximum. (An open packing $P$ is {\em maximal} if after adding any new vertex to $P$ the resulting set is not an open packing.)
 An analogous concept for standard graph coloring is already established. Namely, a {\em fall coloring} of a graph $G$, as introduced by Dunbar et al.~\cite{DUNBAR}, is a partition of $V(G)$ into maximal independent sets. The concept is also known as {\em idomatic partition}~\cite{km,VALENCIA}. Note that a graph $G$ need not have a fall coloring, and if it does, the minimum number of colors required for such a coloring is the {\em fall chromatic number} of $G$.  A naturally interesting case is when the fall chromatic number equals the chromatic number of a graph. Among variois studies of fall colorings, we mention two recent papers considering their complexity issues~\cite{CAMPOS, LAURI}

 Taking this terminology into account, call a partition of the vertex set of a graph $G$ into maximal open packings an {\em injective fall coloring}. If a graph $G$ admits such a partition, then $G$ is an {\em injectively fall colorable graph}. If, in addition, $G$ admits an injective fall coloring with $\chi_i(G)$ colors, then $G$ is an {\em injectively fall $\chi_i(G)$-colorable graph}. Clearly, every perfect injectively colorable graph is an injectively fall $\chi_i(G)$-colorable graph, but the converse is not true; take the graph house as a small example. Also, every injectively fall $\chi_i(G)$-colorable graph is an injectively fall colorable graph. While we do not know if the converse is also true, we prove that these two classes coincide within diameter $2$ graphs.

\begin{proposition}\label{prp:IFCdiam2}
Let $G$ be a graph with $\diam(G)=2$. The following statements are equivalent:
\begin{enumerate}[(i)]
\item $G$ is an injectively fall $\chi_i(G)$-colorable graph;
\item $G$ is an injectively fall colorable graph;
\item there exists a matching $M$ in $G$ such that each edge of $M$ does not belong to a triangle, and every edge incident with a vertex not in $V(M)$ lies in a triangle.
\end{enumerate}
\end{proposition}
\begin{proof}
The direction (i)$\implies$(ii) is trivial. To see (ii)$\implies$(iii), let $G$ be an injectively fall colorable graph, and let $\mathbb{P}=\{P_{1},\dots,P_{|\mathbb{P}|}\}$ be the color partition of an injective fall coloring of $G$. Since $G$ has diameter $2$, each $P_i$ either consists of a vertex or of two adjacent vertices. If $P_i$ consist of two adjacent vertices $u$ and $v$, then, since $P_i$ is an open packing, the edge $uv$ does not lie in a triangle. Thus the edges that correspond to sets $P_i$ with $|P_i|=2$ form a matching $M$ in $G$ as stated by condition (iii). On the other hand, if $|P_i|=1$ with $P_i=\{u\}$, then $P_i$ being maximal implies that any edge $uv$ must be contained in a triangle for otherwise $\{u,v\}$ would be an open packing containing $P_i$. Thus (iii) is proved.

To prove (iii)$\implies$ (i), first note that the partition of $V(G)$, $\mathbb{P}=\{\{u,v\}:\, uv\in M\}\cup\{\{w\}:\, w\textrm{ is not incident with an edge of } M\}$ yields an injective coloring of $G$. Indeed, since edges of $M$ do not lie in a triangle, their end-vertices form an open packing, while the singletons always form an open packing. Thus, $\mathbb{P}$ is an open packing partition of $G$. Since diam$(G)=2$, each $\{u,v\}\in \mathbb{P}$ is clearly a maximal open packing in $G$. Let $\{w\}\in \mathbb{P}$. If $\{w\}$ is not a maximal open packing, then there exists an open packing $\{w,z\}$ in $G$. Moreover, $wz\in E(G)$ as diam$(G)=2$. Now, $wz$ lies in a triangle by the hypothesis, which contradicts the fact that $\{w,z\}$ is an open packing. Consequently, $G$ is injectively fall colorable. To see that $|\mathbb{P}|=\chi_i(G)$, note that every vertex, which is not incident with $M$, receives a unique color in any injective coloring of $G$, since it is adjacent to all other vertices in ${\cal N}(G)$. Hence, $\chi_i(G)\ge |V(G)\setminus V(M)|+|V(M)|/2=|V(G)|-|V(M)|+|V(M)|/2=|\mathbb{P}|$. Thus $G$ is an injectively fall $\chi_i(G)$-colorable graph.
\end{proof}

We continue with analyzing the case of even cycles with respect to containment in one of the three classes of graphs that we study in this section. To do so, we recall that the \textit{lower open packing number} of $G$, denoted $\rho_{L}^{o}(G)$, is the minimum cardinality of a maximal open packing of $G$. We make use of the next lemma. Moreover, w observe that $n\equiv2$ (mod $4$) and $n\equiv0$ (mod $3$) is equivalent to $n\equiv6$ (mod $12$).

\begin{lemma}\emph{(\cite{hs})}\label{bbi}
Let $n\geq3$. Then $\rho_{L}^{o}(C_{n})=\lceil\frac{n}{3}\rceil+1$ if $n\equiv2$ \emph{(}mod $6$\emph{)}, and $\rho_{L}^{o}(C_{n})=\lceil\frac{n}{3}\rceil$ otherwise.
\end{lemma}

\begin{proposition}\label{prp:PICcycles}
Let $C_n$ be an even cycle. Then, it is injectively fall colorable if and only if either $n\equiv0\pmod 4$ or $n\equiv6$ \emph{(mod $12$)}. In addition, $C_n$ is perfect injectively colorable if and only if $n\equiv 0\pmod 4$ or $n=6$. Also, it is injectively fall $\chi_i(C_{n})$-colorable if and only if $n\equiv0\pmod 4$, or $n\equiv6$ \emph{(mod $12$)}.
\end{proposition}

\begin{proof}
We start the proof by noting that $\opack(C_n)=2\lfloor \frac{n}{4}\rfloor$ for any even integer $n$, and $\chi_i(C_n)=2$ if $n\equiv 0 \pmod 4$, while $\chi_i(C_n)=3$ if $n\equiv 2 \pmod 4$.

Let $n\equiv 0 \pmod 4$. Note that the pattern ``$1122\ldots$'' repeated along the vertices of the cycle yields an injective $2$-coloring of $C_{n}$, in which each color class has cardinality $\frac{n}{2}=\opack(C_n)$. Thus $C_n$ is perfect injectively colorable (in particular, it is both injectively fall $\chi_i(C_{n})$-colorable and injectively fall colorable) in this case.

From now on, let $n\equiv 2 \pmod 4$. First observe that $C_6$ is a perfect injectively colorable graph. Next, we write $n=4q+2$ for some integer $q\geq 2$. Since $\opack(C_n)=2\lfloor \frac{n}{4}\rfloor=2q$ and $\chi_i(C_n)=3$, the equality $\opack(C_n)\chi_{i}(C_n)=n$ does not hold, which implies that $C_n$ is not perfect injectively colorable graph.

Now, let $n\equiv6$ (mod $12$). Then the pattern ``$123\ldots$'' repeated along the vertices of the cycle yields an injective $\chi_i(C_{n})$-coloring of $C_n$, in which each color class is a maximal open packing. Thus, $C_n$ is is both injectively fall $\chi_i(C_{n})$-colorable and injectively fall colorable, as claimed.

It remains to consider the case when $n$ is not congruent to $6$ modulo $12$ \big{(}and $n\equiv2$ (mod $4$)\big{)}. We claim that in this case, regardless of how an injective $k$-coloring ($k\ge 3$) of $C_n$ is constructed, at least one of the color classes will not be a maximal open packing of $C_n$. Let $f$ be an injective $k$-coloring function of $C_{n}$ for which any color class is maximal. Let $X_{i}$ be the set of vertices with color $i$ under $f$. Since $n$ is not congruent to $6$ modulo $12$ and $n\equiv2$ (mod $4$), it follows that $n=12q+2$ or $n=12q'+10$ for some integers $q,q'\geq1$. If $n=12q+2$, then $\rho_{L}^{o}(C_{n})=4q+2$ by Lemma \ref{bbi}. So, $12q+2=|X_{1}|+\cdots+|X_{k}|\geq k(4q+2)\geq3(4q+2)$, a contradiction. Moreover, $\rho_{L}^{o}(C_{n})=4q'+4$ if $n=12q'+10$. So, $12q'+4=|X_{1}|+\cdots+|X_{k}|\geq k(4q'+4)\geq3(4q'+4)$, which is again a contradiction. Therefore, $C_{n}$ is not injectively fall colorable. This also implies that $C_{n}$ is not injectively fall $\chi_{i}(C_{n})$-colorable. This completes the proof.
\end{proof}

In the rest of this section, we focus on the family of hypercubes $\{Q_n\}_{n\in\mathbb{N}}$, where $Q_1=K_2$, and $Q_n=Q_{n-1}\cart K_2$ for $n\ge 2$.  Recall that a {\em perfect code} (also called a 1-perfect code) in a graph $G$ is a set which is at the same time a packing and a dominating set. A  {\em code-coloring} of $G$ is a partition of $V(G)$ into perfect codes. Clearly, not all graphs admit a perfect code let alone a code-coloring, where for a simple example one can take $Q_2=C_4$. It is also easy to see that any two antipodal vertices of $Q_3$ form a perfect code, and thus one can find a code-coloring in the $3$-cube. Note that each color class of a code-coloring of a graph $G$ presents a maximum packing, that is, a $\rho(G)$-set of $G$.

First, we give a general observation about the specific construction of the partition into $\opack(G)$-sets of prisms, which is based on the existence of code-colorings.

\begin{lemma}
\label{lem:prism}
If a graph $G$ has a code-coloring, then $G\cart K_2$ is a perfect injectively colorable graph.
\end{lemma}
\begin{proof}
Let $G$ be a graph and $\mathbb{P}=\{P_{1},\dots,P_{|\mathbb{P}|}\}$ a code-coloring of $G$. By definition, for any $i$, the set $P_i$ is a maximum packing and a dominating set in $G$. Clearly, $P_i\times V(K_2)$ is an open packing and at the same time a total dominating set in $G\cart K_2$. We infer that $P_i\times V(K_2)$ is an $\opack$-set of $G\cart K_2$. Consequently, $\mathbb{Q}=\{P_{1}\times V(K_2),\dots,P_{|\mathbb{P}|}\times V(K_2)\}$ is a partition into $\opack$-sets of $G \cart K_2$.
\end{proof}

Mollard~\cite{MOL} considered code-colorings of regular graphs, and proved the following result.

\begin{theorem}
\label{thm:mollard}
 {\rm \cite[Theorem 7]{MOL}}
Let $G$ and $H$ be regular graphs of the same degree $n$. If $H$ is bipartite and there exists a code-coloring in $G$ and in $H$, then there exists a code-coloring in $G\cart H\cart K_2$.
\end{theorem}

By inductively applying Theorem~\ref{thm:mollard}, we get that $Q_{2^k-1}$ has a code-coloring for every $k\in \mathbb{N}$, and combining this with Lemma~\ref{lem:prism} we infer the following result.

\begin{theorem}\label{thm:hypercubes}
If $k\in \mathbb{N}$, then $Q_{2^k}$ is a perfect injectively colorable graph.
\end{theorem}

The problem of characterizing the hypercubes, which are perfect injectively colorable or injectively fall colorable, is still open. Beside Theorem~\ref{thm:hypercubes}, we have a few additional sporadic examples. Clearly, $Q_3$ is a perfect injectively colorable graph, and one can also check that $Q_5$ is as well. We suspect that this happens also with $Q_6$, since the set
$$S=\{000000,000001,001110,001111,110110,110111,111000,111001\},$$
given in a binary representation of $Q_6$ is a $\rho^o(Q_6)$-set of cardinality $8$, which leads to a partition of $V(Q_6)$ into $\opack(Q_6)$-sets. These facts bring the following question.

\begin{problem}
For which positive integers $n$, is the hypercube $Q_n$ a perfect injectively colorable graph \emph{(}an injectively fall colorable graph, respectively\emph{)}?
\end{problem}

Clearly, an answer to the question above passes first through computing the value of the open packing number of hypercubes. Hence, the following question is also worthwhile.

\begin{problem}
Which is the value of $\rho^o(Q_n)$ for any $n\ge 6$?
\end{problem}

We conclude this section with some other natural open problems.
Note that Propositions~\ref{prp:PICdiam2} and~\ref{prp:IFCdiam2} provide characterizations of the three classes among diameter $2$ graphs, that is, the graphs $G$ with $\opack(G)\le 2$. The next small step in characterizing the three classes of graphs is presented in the following problem.

\begin{problem}
Characterize the perfect injectively colorable graphs, the injectively $\chi_i(G)$-fall colorable graphs, and the injectively fall colorable graphs $G$, respectively, among graphs $G$ with $\opack(G)=3$.
\end{problem}

As noted in Proposition~\ref{prp:IFCdiam2}, the classes of injectively $\chi_i(G)$-fall colorable graphs and injectively fall colorable graphs coincide in graphs $G$ with diameter $2$. Actually, we do not know if the two classes differ, though we suspect they do. Anyway, we pose it as the following question.

\begin{problem}
Does there exist a graph $G$, which is injectively fall colorable, but not injectively $\chi_i(G)$-fall colorable?
\end{problem}


\section*{Acknowledgments}

B.B. was supported by the Slovenian Research Agency (ARRS) under the grants P1-0297, J1-2452 and J1-3002.
I.G.Y. has been partially supported by ``Junta de Andaluc\'ia'', FEDER-UPO Research and Development Call, reference number UPO-1263769. Moreover, this investigation was completed while this last author was visiting the University of Ljubljana, Slovenia, supported by ``Ministerio de Educaci\'on, Cultura y Deporte'', Spain, under the ``Jos\'e Castillejo'' program for young researchers (reference number: CAS21/00100).


\end{document}